\documentclass[11pt]{amsart}

\usepackage{amssymb}
\usepackage{mathtools}
\usepackage{enumitem}
\usepackage{hyperref}
\usepackage{tikz}

\usepackage{fullpage}

\newtheorem{theorem}{Theorem}[section]
\newtheorem{lemma}[theorem]{Lemma}

\newtheorem{corollary}[theorem]{Corollary}
\newtheorem{conjecture}[theorem]{Conjecture}

\newtheorem{alphtheorem}{Theorem}

\theoremstyle{definition}
\newtheorem*{ack}{Acknowledgments}
\newtheorem*{con}{Conventions}

\newtheorem{example}[theorem]{Example}
\newtheorem{definition}[theorem]{Definition}

\DeclarePairedDelimiter\ceil{\lceil}{\rceil}
\DeclarePairedDelimiter\floor{\lfloor}{\rfloor}

\DeclareMathOperator\id{id}

\DeclareMathOperator\pr{pr}

\DeclareMathOperator\Spec{Spec}
\DeclareMathOperator\Sym{Sym}

\DeclareMathOperator\supp{supp}

\DeclareMathOperator\an{an}

\DeclareMathOperator\HH{H}

\newcommand*\ratmap{\mathbin{\tikz [baseline=0ex,-latex, dashed, ->] \draw [densely dashed] (0em,0.58ex) -- (1.3em,0.58ex);}}

\renewcommand{\epsilon}{\varepsilon}

\title{The map to the orbifold base need not be an orbifold map}

\author{Finn Bartsch}
\address{Finn Bartsch \\
IMAPP Radboud University Nijmegen \\
PO Box 9010, 6500GL \\
Nijmegen, The Netherlands\\}
\email{f.bartsch@math.ru.nl}


\begin{document}

\begin{abstract}
We give an explicit example of a fibration $f \colon X \to Y$ between smooth projective varieties whose ``orbifold base'' $\Delta_f$ in the sense of Campana has the property that the induced morphism $X \to (Y, \Delta_f)$ is not a morphism of C-pairs (i.e.\ is not an ``orbifold morphism'').
We however also show that this cannot happen if $f$ is ``neat'' and $(Y, \Delta_f)$ is sufficiently well-behaved.
Finally, we discuss the implications of this statement towards conjectures of Campana aiming to give algebro-geometric characterizations of those varieties which either admit a dense entire curve or a potentially dense set of integral points.
\end{abstract}

\maketitle
\thispagestyle{empty}

\section{Introduction}

In Campana's theory of \emph{special varieties} and \emph{orbifolds}, pairs $(X, \Delta)$, where $X$ is a variety and $\Delta$ is an effective $\mathbb{Q}$-divisor on $X$ with standard coefficients, appear in two related but distinct roles.\smallbreak

On the one hand, there is the notion of the \emph{orbifold base} of a dominant morphism $f \colon X \to Y$, introduced in \cite[Definition~1.5]{CampanaFourier}, which gives a way of encoding the nowhere reduced fibers of $f$ in terms of a $\mathbb{Q}$-divisor on $Y$.
For smooth varieties $X$ and $Y$ and a dominant morphism $f \colon X \to Y$, the orbifold base of $f$ can be defined as follows.
Given a prime divisor $D \subseteq Y$, consider the pullback $f^*D$ and decompose it as $f^*D = \sum_i a_i D_i + R$, where the $a_i$ are positive integers, the $D_i$ are prime divisors which dominate $D$ via $f$, and $R$ is an effective divisor on $X$ whose components do not dominate $D$ via $f$.
Then denote by $m(f,D)$ the infimum of the $a_i$ (with the convention that the infimum of the empty set is $\infty$) and define $\Delta_f := \sum_D (1-\frac{1}{m(f,D)}) D$ (with the convention that $\frac{1}{\infty} = 0$).
Since $m(f,D) = 1$ unless $D$ is not in the image of $f$ or $f$ is not smooth over $D$, we see that $m(f,D) = 1$ for all but finitely many prime divisors $D \subseteq Y$ (at least in characteristic zero, or when $f$ is generically separable) and thus, $\Delta_f$ is indeed a $\mathbb{Q}$-divisor.
We call $\Delta_f$ the \emph{orbifold base divisor} of $f$ and the pair $(Y, \Delta_f)$ is called the \emph{orbifold base} of $f$.
For example, if $f$ is an elliptic fibration $X \to C$ over a curve $C$, then the orbifold base in the sense of Campana coincides with the classical notion of the ``orbifold base of an elliptic fibration'' which encodes the multiple fibers of $f$.\smallbreak

On the other hand, a different appearance of pairs $(X, \Delta)$ in Campana's theory comes from \emph{C-pairs}.
These objects and their morphisms were first properly introduced in \cite[Définition~2.3]{CampanaJussieu} under the name ``orbifoldes géométriques''.
A C-pair simply consists of a smooth variety $X$ and a $\mathbb{Q}$-divisor $\Delta$ on $X$ of the form $\Delta = \sum_i (1-\frac{1}{m_i}) D_i$ with $m_i \in \mathbb{Z}_{\geq 2} \cup \{ \infty \}$; the number $m_i$ is called the \emph{multiplicity} of $D_i$ in $\Delta$.
The interest of C-pairs comes from how a morphism into them is defined.
A morphism $T \to (X, \Delta)$ from a smooth variety $T$ is a morphism $f \colon T \to X$ whose image is disjoint from $\floor{\Delta}$, such that for every prime divisor $D \subseteq \supp \Delta$, either $f$ factors through $D$ or the coefficients of the pullback $f^*D$ are all at least the multiplicity of $D$ in $\Delta$.
(We note that this slightly differs from Campana's original definition in that we allow $f$ to factor over components of $\Delta$.
In the case of dominant morphisms, this makes no difference.)
For example, if $C \subseteq \mathbb{P}^2$ is a plane curve and $L \subseteq \mathbb{P}^2$ is a line, the inclusion morphism $C \to \mathbb{P}^2$ is a C-pair morphism $C \to (\mathbb{P}^2, \frac{1}{2}L)$ if and only if every intersection of $C$ and $L$ is tangential.\smallbreak

Given these two notions, it is easy to see that given a smooth variety $X$, a C-pair $(Y, \Delta_Y)$ and a dominant morphism $f \colon X \to (Y, \Delta_Y)$, we have $\Delta_Y \leq \Delta_f$, where $\Delta_f$ denotes the orbifold base of the underlying morphism of varieties $X \to Y$.
On the other hand, it is natural to wonder if the converse holds, that is, whether given a dominant morphism of smooth varieties $f \colon X \to Y$, it induces a C-pair morphism $X \to (Y, \Delta_f)$ onto the orbifold base.
Unravelling the definitions, we see that this is true if and only if the coefficients of the divisor $R$ appearing in the definition of the orbifold base are all at least the infimum of the $a_i$ appearing there.
In particular we do obtain a C-pair morphism $X \to (Y, \Delta_f)$ if we always have $R = 0$; that is, if $f$ does not contract any divisors of $X$ into a codimension at least two subset of $Y$.
This also implies that if $Z \subseteq X$ is the union of the finitely many divisors contracted by $f$, we do obtain a C-pair morphism $X \setminus Z \to (Y, \Delta_f)$.
Since flat morphisms do not contract any divisors, we hence see that if $f \colon X \to Y$ is flat, then $f$ does indeed give a C-pair morphism $X \to (Y, \Delta_f)$.
The morphism $f \colon \mathbb{A}^2 \to \mathbb{A}^2$ given by $f(x,y) = (x,xy)$, whose orbifold base divisor is $\{ x=0 \}$, shows that some assumption of this form is also necessary.
However, since this example very much relies on $f$ being non-surjective, one might hope that at least under some additional assumptions, e.g.\ for proper surjective morphisms between smooth varieties, the flatness assumption can be dropped.
The first purpose of this note is to give concrete examples which show that this is not possible.

\begin{alphtheorem}[Section~\ref{sect:examples}] \label{mainthm}
There exists a surjective morphism $f \colon X \to Y$ from a smooth projective threefold $X$ onto a smooth projective surface $Y$, with connected fibers, such that $X \to (Y, \Delta_f)$ is not a C-pair morphism.
\end{alphtheorem}

The possibility of a morphism as in Theorem~\ref{mainthm} existing has been noticed before; see for example \cite[Remark~3.26]{JRGeomSpec} or \cite{CampanaErratum}.
However, the constructions we give below seem to be the first explicit examples of this behaviour.\smallbreak

In practice, one way C-pairs are useful is that morphisms to a C-pair yield constraints on the entire curves lying on a variety $X$ defined over $\mathbb{C}$.
An \emph{entire curve} on $X$ is a holomorphic morphism $\mathbb{C} \to X^{\an}$ (where $X^{\an}$ denotes the complex-analytic space associated to the variety $X$).
Our definition of a morphism into a C-pair immediately adapts to the complex-analytic setting, so that we also have a good notion of an entire curve on a C-pair $(Y, \Delta)$.
If now $X \to (Y, \Delta)$ is a C-pair morphism and $\mathbb{C} \to X$ is an entire curve in $X$, then the composition $\mathbb{C} \to (Y, \Delta)$ will define an entire curve in $(Y, \Delta)$.
If we now additionally know that $(Y, \Delta)$ does not admit any nonconstant entire curves, then this forces every entire curve on $X$ to be contained in a fiber of $X \to Y$.
This is for example the case if $(Y, \Delta) = (\mathbb{P}^1, \frac{1}{2}([a_1]+\ldots+[a_5]))$ with $a_1,\ldots,a_5$ five pairwise distinct points on $\mathbb{P}^1$ \cite[§12, Theorem~4]{CampanaWinkelmann}.
We refer to \cite{CampanaPaunBTCP} or \cite[Section~6]{LafonThreefoldPaper} for explicit examples of this line of reasoning.
In theory, similar reasoning can be used to constrain the rational or integral points lying on a variety $X$ defined over a number field, but in practice this is hard to do unconditionally due to our lack of understanding of integral points on C-pairs.
Conditionally on the abc conjecture, this is done in \cite{LafonThreefoldPaper}.\smallbreak

In \cite[Section~9]{CampanaFourier}, Campana conjectures that a smooth projective variety admits a Zariski-dense entire curve if and only if it is \emph{special} in the sense of \cite[Definition~2.1]{CampanaFourier}.
Since the word ``special'' is somewhat overloaded in the literature, we instead choose to call these varieties \emph{Campana-special} in the present text.
A smooth projective variety $X$ fails to be Campana-special if and only if there is a proper birational morphism $X' \to X$ with $X'$ smooth, and a \emph{neat} surjective morphism $f \colon X' \to Y$ with connected fibers for which the pair $(Y, \Delta_f)$ is of general type.
Here, we say that a dominant morphism $X \to Y$ of smooth varieties is \emph{neat} if there is a proper birational morphism $X \to X''$ with $X''$ smooth such that every divisor contracted by $X \to Y$ is also contracted by $X \to X''$.
One would hope that, by some C-pair version of the Green--Griffiths--Lang conjecture, the general type pair $(Y, \Delta_f)$ does not admit a dense entire curve and one would now hope that this knowledge helps with constraining the entire curves on $X'$ and hence on $X$ -- in this way, one could conclude that a variety admitting a dense entire curve is Campana-special (assuming Green--Griffiths--Lang for general type C-pairs).
This step however requires the morphism $X' \to (Y, \Delta_f)$ to be a C-pair morphism, which by Theorem~\ref{mainthm} is not completely trivial.
The second purpose of this note is to show that for \emph{neat} morphisms whose orbifold base divisor has snc support, there is no problem.

\begin{alphtheorem}[Section~\ref{sect:neatness}] \label{mainthmb}
Let $f \colon X \to Y$ be a dominant morphism of smooth varieties and let $\Delta_f$ be its orbifold base divisor.
Assume that $Y$ is quasi-projective, that $f$ is neat, and that $\Delta_f$ is supported on a divisor with simple normal crossings.
Then $f$ is a C-pair morphism $X \to (Y, \Delta_f)$.
\end{alphtheorem}

Finally, in Section~\ref{sect:applications}, we make the above discussion precise and explain in detail how C-pair analogues of the Green--Griffiths--Lang and the Bombieri--Lang conjectures imply conjectures of Campana regarding the distribution of dense entire curves and integral points on varieties.
Our discussion can be summarized as follows.

\begin{alphtheorem}[Section~\ref{sect:applications}] \label{mainthmc}
Assume that the weak Green--Griffiths--Lang conjecture holds for C-pairs of dimension $\leq n-1$ and for varieties of dimension $\leq n$.
Let $X$ be a variety over $\mathbb{C}$ of dimension $\leq n$ admitting a holomorphic map $\mathbb{C} \to X^{\an}$ with Zariski-dense image.
Then any resolution of singularities of $X$ is Campana-special.
\end{alphtheorem}

\begin{alphtheorem}[Section~\ref{sect:applications}] \label{mainthmd}
Assume that the weak Bombieri--Lang conjecture holds for C-pairs of dimension $\leq n-1$ and for varieties of dimension $\leq n$.
Let $X$ be a variety over a subfield of $\overline{\mathbb{Q}}$ of dimension $\leq n$ admitting a potentially dense set of integral points.
Then any resolution of singularities of $X$ is Campana-special.
\end{alphtheorem}

In forthcoming work with Javanpeykar, we will show that the converse implications in Theorems~\ref{mainthmc} and \ref{mainthmd} hold as well.
(That is, if every variety admitting a Zariski-dense entire curve is Campana-special, then weak Green--Griffiths--Lang for C-pairs holds.)

\begin{con}
We work over an algebraically closed base field $k$ of characteristic zero.
Varieties are finite type separated integral schemes over $k$.
\end{con}

\begin{ack}
I thank Ariyan Javanpeykar for his constant support and many helpful discussions.
\end{ack}

\section{Examples} \label{sect:examples}

In this section, we give examples of morphisms of varieties which do not induce a C-pair morphism to their orbifold base.
The example mentioned in Theorem~\ref{mainthm} is Example~\ref{threefold_to_surface:example} below.
Let us however first start with an example where both the source and the target are smooth surfaces and the map between them is generically finite of degree two.

\begin{example} \label{surface_to_surface:example}
Let $C$ be a smooth quasi-projective curve, let $\sigma \colon C \to C$ be an involution with at least one fixpoint and let $\overline{C}$ be the corresponding (scheme-theoretic) quotient of $C$.
Consider the surface $C \times C$ with its involution $(\sigma, \sigma)$ and let $X$ be the quotient $(C \times C)/(\sigma, \sigma)$.
Let $\widetilde{X} \to X$ be the minimal resolution of singularities (obtained by blowing up once in each of the singular points).
There is a natural map $X \to \overline{C} \times \overline{C}$ which yields a morphism $f \colon \widetilde{X} \to \overline{C} \times \overline{C}$.
\end{example}

Observe that the morphism $f$ in this example is not flat: while the morphism $X \to \overline{C} \times \overline{C}$ is finite and hence flat by miracle flatness, precomposing with the blowup $\widetilde{X} \to X$ introduces positive-dimensional fibers.
More precisely, the morphism $f$ contracts some divisors of $\widetilde{X}$ into codimension two subsets of $\overline{C} \times \overline{C}$.
We first show that the orbifold base of this morphism is non-trivial.

\begin{lemma} \label{surface_to_surface:orbifold_base1}
In the situation of Example~\ref{surface_to_surface:example}, the orbifold base divisor $\Delta_f$ of $f$ is given by $\frac{1}{2}(\Delta \times \overline{C} + \overline{C} \times \Delta)$, where $\Delta \subseteq \overline{C}$ is the branch locus of the quotient map $C \to \overline{C}$.
\end{lemma}
\begin{proof}
Consider first two points $\overline{c},\overline{d} \in \overline{C} \setminus \Delta$.
Then the quotient map $C \to \overline{C}$ is étale over $\overline{c}$ and $\overline{d}$ and it follows that the morphism $f$ is finite étale over a neighborhood of $(\overline{c}, \overline{d})$.
In particular, the point $(\overline{c}, \overline{d})$ cannot be contained in the support of the orbifold base divisor of $f$ and thus, the orbifold base divisor of $f$ must be supported on $\Delta \times \overline{C} + \overline{C} \times \Delta$.
Now let $\overline{c} \in \Delta$ be a point over which the quotient map $C \to \overline{C}$ ramifies.
We have to show that the divisor $\{ \overline{c} \} \times \overline{C}$ appears in the support of the orbifold base divisor with multiplicity two (by symmetry, this suffices to prove the claim).
To do so, since the definition of the orbifold base divisor asks us to discard any divisors that get contracted under $f$, we may restrict to a Zariski-open neighborhood of the generic point of $\{ \overline{c} \} \times \overline{C} \subseteq \overline{C} \times \overline{C}$.
Thus, to compute the multiplicity of $\{ \overline{c} \} \times \overline{C}$ in $\Delta_f$, we may replace $\widetilde{X}$ by the preimage of $\overline{C} \times (\overline{C} \setminus \Delta)$.
Note that the quotient variety $X$ is smooth over $\overline{C} \times (\overline{C} \setminus \Delta)$, so that the blowup $\widetilde{X} \to X$ is an isomorphism over $\overline{C} \times (\overline{C} \setminus \Delta)$.
Since the morphism $X \to \overline{C} \times \overline{C}$ is finite of degree $2$ and ramifies over $\{ \overline{c} \} \times \overline{C} \subseteq \overline{C} \times \overline{C}$, the claim now follows.
\end{proof}

In the proof, we crucially used that we were allowed to not consider the exceptional divisors of the blowup $\widetilde{X} \to X$ when computing the multiplicities, since they are contracted to a point in $\overline{C} \times \overline{C}$, which is a surface.
On the other hand, if we now postcompose our morphism $f$ with a projection map $\overline{C} \times \overline{C} \to \overline{C}$, the exceptional divisors of $\widetilde{X} \to X$ still get contracted to a point in $\overline{C}$ -- but a point is now of codimension one, so we cannot ignore it anymore and thus must take the exceptional divisors into account.
The end result is that after doing this postcomposition, the orbifold base divisor becomes trivial.

\begin{lemma} \label{surface_to_surface:orbifold_base2}
In the situation of Example~\ref{surface_to_surface:example}, consider the projection onto the first factor $\pr_1 \colon \overline{C} \times \overline{C} \to \overline{C}$. 
Then the orbifold base of the composition $\pr_1 \circ f \colon \widetilde{X} \to \overline{C}$ is trivial (i.e.\ $\Delta_{\pr_1 \circ f} = 0$).
\end{lemma}
\begin{proof}
If $c_0 \in C$ is a fixpoint of the involution $\sigma$, the morphism $C \to C \times C$ given by $c \mapsto (c, c_0)$ is equivariant for the involutions $\sigma$ and $(\sigma, \sigma)$, respectively.
Thus, it descends to the respective quotients and we obtain a morphism $\overline{C} \to X$.
Since $\overline{C}$ is a smooth curve and $\widetilde{X} \to X$ is proper, this morphism lifts to give a morphism $\overline{C} \to \widetilde{X}$.
By construction, this yields a section of $\pr_1 \circ f$.
Since $\widetilde{X}$ is smooth, the existence of a section implies that every fiber of $\pr_1 \circ f$ has a reduced component; hence we are done. 
\end{proof}

From these two facts, it now follows almost formally that $\widetilde{X} \to (\overline{C} \times \overline{C}, \Delta_f)$ cannot be a morphism of C-pairs.
However, for a clean statement of the argument, it will be convenient to extend our notion of morphisms into C-pairs to also cover morphisms between C-pairs.
To do so, let $(X, \Delta_X)$ and $(Y, \Delta_Y)$ be two C-pairs, with $\Delta_X = \sum_i (1-\frac{1}{m_i}) E_i$ and $\Delta_Y = \sum_j (1-\frac{1}{m_j'}) E_j'$.
Then a morphism $f \colon X \to Y$ is said to be a \emph{morphism of C-pairs} $(X, \Delta_X) \to (Y, \Delta_Y)$ if $f^{-1}(\floor{\Delta_Y}) \subseteq \floor{\Delta_X}$ and for every $j$, we either have that $f$ factors over $E_j'$ or that for every prime divisor $D \subseteq X$ appearing in $f^*E_j'$, the coefficient of $D$ in $f^*E_j'$ is at least $\frac{m_j'}{m_i}$ if $D = E_i$ for some $i$ and at least $m_j'$ otherwise.
Note that a morphism of C-pairs $(X,0) \to (Y, \Delta_Y)$ is just a morphism $X \to (Y, \Delta_Y)$ as defined before.
Note furthermore that this notion is stable under composition:
Given C-pair morphisms $(X, \Delta_X) \to (Y, \Delta_Y)$ and $(Y, \Delta_Y) \to (Z, \Delta_Z)$, the composition $X \to Z$ defines a C-pair morphism $(X, \Delta_X) \to (Z, \Delta_Z)$.

\begin{corollary}
In the situation of Example~\ref{surface_to_surface:example}, the morphism $f$ does not give a C-pair morphism $\widetilde{X} \to (\overline{C} \times \overline{C}, \Delta_f)$.
\end{corollary}
\begin{proof}
Let $\Delta \subseteq \overline{C}$ denote the branch locus of the quotient morphism $C \to \overline{C}$.
Then it follows from Lemma~\ref{surface_to_surface:orbifold_base1} that projection onto the first component defines a C-pair morphism $(\overline{C} \times \overline{C}, \Delta_f) \to (\overline{C}, \frac{1}{2} \Delta)$.
Consequently, if $f \colon \widetilde{X} \to (\overline{C} \times \overline{C}, \Delta_f)$ were to be a C-pair morphism, the composition $\pr_1 \circ f$ would be a C-pair morphism $\widetilde{X} \to (\overline{C}, \frac{1}{2} \Delta)$.
But this in turn would force the orbifold base of $\pr_1 \circ f$ to be nontrivial, contradicting Lemma~\ref{surface_to_surface:orbifold_base2}.
\end{proof}

The second example is closely related to the first, but this time, we have a smooth threefold mapping to a smooth surface and the map between them has connected fibers. 

\begin{example} \label{threefold_to_surface:example}
Let $C$ be a smooth quasi-projective curve, let $\sigma \colon C \to C$ be an involution with at least one fixpoint and let $\overline{C}$ be the corresponding (scheme-theoretic) quotient of $C$.
Let $D$ be a smooth quasi-projective curve and let $\tau \colon D \to D$ be a fixpoint free involution.
Consider the threefold $C \times C \times D$ and its automorphisms $(\sigma, \sigma, \id_D)$ and $(\id_C, \sigma, \tau)$.
They generate a group of order four (isomorphic to the Klein four group), which we call $G$.
Let $X$ be the quotient $(C \times C \times D)/G$.
Let $\widetilde{X} \to X$ be a resolution of singularities which is an isomorphism over the regular locus of $X$.
The projection map $C \times C \times D \to C \times C$ induces a morphism $X \to \overline{C} \times \overline{C}$ and composing it yields a morphism $f \colon \widetilde{X} \to \overline{C} \times \overline{C}$.
\end{example}

The proof that the morphism $f$ indeed is the desired example is very similar to the one given in Lemmas~\ref{surface_to_surface:orbifold_base1} and \ref{surface_to_surface:orbifold_base2}.

\begin{lemma} \label{threefold_to_surface:orbifold_base1}
In the situation of Example~\ref{threefold_to_surface:example}, the orbifold base divisor $\Delta_f$ of $f$ is given by $\frac{1}{2}(\Delta \times \overline{C} + \overline{C} \times \Delta)$, where $\Delta \subseteq \overline{C}$ is the branch locus of the quotient map $C \to \overline{C}$.
\end{lemma}
\begin{proof}
We first prove that the image of the singular locus of $X$ in $\overline{C} \times \overline{C}$ is contained in $\Delta \times \Delta$ and is thus, in particular finite.
To do so, note that the singular locus of $X$ lies under the points of $C \times C \times D$ with nontrivial stabilizer.
Since the involution $\tau$ is assumed to have no fixed points, the only element of $G$ with fixed points is $(\sigma, \sigma, \id_D)$.
The image of the fixed points of $(\sigma, \sigma, \id_D)$ in $\overline{C} \times \overline{C}$ is given by $\Delta \times \Delta$ and the claim follows.

Given a point $(\overline{c_1}, \overline{c_2}) \in \overline{C} \times \overline{C}$ with neither of the $\overline{c_i}$ contained in $\Delta$, the preimage of $(\overline{c_1}, \overline{c_2})$ in $C \times C \times D$ consists of four copies of $D$ which are interchanged by the action of the group $G$.
In particular, it is reduced.
Since the map $C \times C \times D \to \overline{C} \times \overline{C}$ can be factored as $C \times C \times D \to X \to \overline{C} \times \overline{C}$, this forces the fiber of $X \to \overline{C} \times \overline{C}$ over $(\overline{c_1}, \overline{c_2})$ to be reduced, too.
Thus, $(\overline{c_1}, \overline{c_2})$ is not contained in the support of $\Delta_f$, so that $\Delta_f$ is supported on $(\Delta \times \overline{C} + \overline{C} \times \Delta)$.

On the other hand, given a point $(\overline{c_1}, \overline{c_2}) \in \Delta \times (\overline{C} \setminus \Delta)$, the preimage of $(\overline{c_1}, \overline{c_2})$ in $C \times C$ consists of two nonreduced points of multiplicity two, which forces the preimage of $(\overline{c_1}, \overline{c_2})$ in $C \times C \times D$ to be nonreduced of multiplicity two, as well.
Since the morphism $C \times C \times D \to X$ is finite étale at points mapping to $(\overline{c_1}, \overline{c_2})$, the fiber of $X \to \overline{C} \times \overline{C}$ over $(\overline{c_1}, \overline{c_2})$ must also be nonreduced of multiplicity two; this shows that $\Delta_f$ contains the term $\frac{1}{2}(\Delta \times \overline{C}$.
By symmetry, $\Delta_f$ also contains the term $\frac{1}{2}(\overline{C} \times \Delta)$; this proves the lemma.
\end{proof}

The analysis done in the previous proof allows one to also quickly deduce the following.

\begin{lemma} \label{threefold_to_surface:connected_fibers}
In the situation of Example~\ref{threefold_to_surface:example}, the morphism $f$ has connected fibers.
\end{lemma}
\begin{proof}
As already remarked in the proof of Lemma~\ref{threefold_to_surface:orbifold_base1}, the preimage of a general point $(\overline{c_1}, \overline{c_2}) \in \overline{C} \times \overline{C}$ in $C \times C \times D$ is given by four copies of $D$, and the action of the group $G$ is transitive on the set of connected components of this preimage.
In particular, passing to the quotient $X = (C \times C \times D)/G$ leaves us with one copy of $D$, which is clearly connected.
Thus, the general fiber of $X \to \overline{C} \times \overline{C}$ is connected.
Since $f$ and $X \to \overline{C} \times \overline{C}$ have the same fibers over the open subset $(\overline{C} \times \overline{C}) \setminus (\Delta \times \Delta)$, we see that the general fiber of $f$ is connected as well.
As $f$ is proper and $X$ is connected, this forces every fiber of $f$ to be connected, as desired.
\end{proof}

\begin{lemma} \label{threefold_to_surface:orbifold_base2}
In the situation of Example~\ref{threefold_to_surface:example}, consider the projection onto the first factor $\pr_1 \colon \overline{C} \times \overline{C} \to \overline{C}$. 
Then the orbifold base of the composition $\pr_1 \circ f \colon \widetilde{X} \to \overline{C}$ is trivial.
\end{lemma}
\begin{proof}
As in the proof of Lemma~\ref{surface_to_surface:orbifold_base2}, it suffices to construct a section of the morphism $X \to \overline{C} \times \overline{C} \to \overline{C}$.
Let $X'$ be the quotient of $C \times C \times D$ by the involution $(\id_C, \sigma, \tau)$.
The involution $(\sigma, \sigma, \id_D)$ of $C \times C \times D$ commutes with $(\id_C, \sigma, \tau)$, and thus it descends to an involution $[\sigma, \sigma, \id_D]$ on $X'$.
Consequently, the quotient morphism $C \times C \times D \to X$ factors over $X'$.
Now let $c_0 \in C$ be a fixed point of $\sigma$ and let $d_0 \in D$ be any point.
Then the morphism $C \to C \times C \times D$ given by $c \mapsto (c, c_0, d_0)$ gives a morphism $C \to X'$ which is equivariant for the involutions $\sigma$ and $[\sigma, \sigma, \id_D]$ respectively.
Thus, it passes to the respective quotients and gives the desired section $\overline{C} \to X$.
\end{proof}

\begin{corollary}
In the situation of Example~\ref{threefold_to_surface:example}, the morphism $f$ does not give a C-pair morphism $\widetilde{X} \to (\overline{C} \times \overline{C}, \Delta_f)$.
\end{corollary}
\begin{proof}
Let $\Delta \subseteq \overline{C}$ denote the branch locus of the quotient morphism $C \to \overline{C}$.
Then it follows from Lemma~\ref{threefold_to_surface:orbifold_base1} that projection onto the first component defines a C-pair morphism $(\overline{C} \times \overline{C}, \Delta_f) \to (\overline{C}, \frac{1}{2} \Delta)$.
Consequently, if $f \colon \widetilde{X} \to (\overline{C} \times \overline{C}, \Delta_f)$ were to be a C-pair morphism, the composition $\pr_1 \circ f$ would be a C-pair morphism $\widetilde{X} \to (\overline{C}, \frac{1}{2} \Delta)$.
But this in turn would force the orbifold base of $\pr_1 \circ f$ to be nontrivial, contradicting Lemma~\ref{threefold_to_surface:orbifold_base2}.
\end{proof}

\section{Neat morphisms} \label{sect:neatness}

In this section, we discuss how neatness helps with making sure that the morphism to the orbifold base is a C-pair morphism, thereby proving Theorem~\ref{mainthmb}.
To formulate our argument, we follow Campana \cite[Section~2.5]{CampanaJussieu} and introduce the sheaf of symmetric differential $1$-forms attached to a smooth C-pair.
Here, we say that a C-pair $(X, \Delta)$ is \emph{smooth} if $\supp \Delta$ is an snc divisor.
If $X$ is a smooth variety and $D \subseteq X$ is a reduced snc divisor, we write $\Omega^1_X(\log D)$ for the sheaf of $1$-forms with at most logarithmic poles along $D$; see \cite[§11.1]{IitakaBook} for its definitions and its basic properties.

\begin{definition} \label{sym_diffs}
Let $(X, \Delta)$ be a smooth C-pair and let $n \geq 1$ be an integer.
Then we define the sheaf of symmetric $1$-forms on $(X, \Delta)$ to be a subsheaf $S^n \Omega^1_{(X, \Delta)} \subseteq \Sym^n \Omega^1_X(\log \supp \Delta)$, which is locally around a closed point $x \in X$ generated as follows.
Choose local coordinates $x_1,\ldots,x_d$ around $x$ such that $\supp \Delta$ is given by the equation $x_1x_2\ldots x_r = 0$ for some $r \leq d$ (which might be zero).
Let $m_i$ denote the multiplicity of the divisor $\{ x_i = 0 \}$ in $\Delta$.
Then the local generators take the following form.
\[ x^{\ceil{v/m}} \bigotimes_{i=1}^n \frac{dx_{j_i}}{x_{j_i}} = x^{\floor{-v(1-1/m)}} \bigotimes_{i=1}^n dx_{j_i} \]
Here, $j_1,\ldots,j_n$ are elements of $\{1,2,\ldots,d\}$, the symbol $v$ denotes a $d$-tuple of integers whose $k$-th entry counts how many how the $j_i$ are equal to $k$, the notation $x^{\ceil{v/m}}$ stands for the product $\prod_{k=1}^d x_k^{\ceil{v_k/m_k}}$, and the notation $x^{\floor{-v(1-1/m)}}$ has the analogous meaning.
\end{definition}
 
This sheaf is defined in such a way that it is well-behaved with respect to C-pair morphisms.

\begin{lemma} \label{sym_diffs_pull_back_cpair_morphism}
Let $X$ be a smooth variety, let $(Y, \Delta)$ be a smooth C-pair and let $f \colon X \to (Y, \Delta)$ be a C-pair morphism not factoring over $\supp \Delta$.
Then, for every integer $n \geq 1$, pullback of differential forms induces a morphism of sheaves $f^* S^n \Omega^1_{(Y, \Delta)} \to \Sym^n \Omega^1_X$.
\end{lemma}
\begin{proof}
This boils down to a straightforward calculation; see \cite[Proposition~2.11.(1)]{CampanaJussieu} or \cite[Lemma~4.3]{CpairKobayashiOchiai} for details.
\end{proof}

The sheaf $S^n \Omega^1_{(X, \Delta)}$ however also has the nice property that the converse of the above lemma is true -- namely, the sheaf can be used to \emph{detect} C-pair morphisms.
The following lemma is similar to \cite[Proposition~2.11.(2)]{CampanaJussieu}.

\begin{lemma} \label{sym_diffs_detect}
Let $(Y, \Delta)$ be a smooth C-pair and let $f \colon X \to Y$ be a morphism of smooth varieties.
Suppose that $f$ does not factor over $\supp \Delta$.
Let $N$ be an integer such that $N \Delta$ is a $\mathbb{Z}$-divisor.
Suppose that pullback of symmetric differentials $f^* \Sym^N \Omega^1_Y \to \Sym^N \Omega^1_X$ extends to a morphism $f^* S^N \Omega^1_{(Y, \Delta)} \to \Sym^N \Omega^1_X$.
Then $f$ is a C-pair morphism $X \to (Y, \Delta)$.
\end{lemma}
\begin{proof}
Let $D \subseteq \supp \Delta$ be an irreducible component, let $m$ be its multiplicity in $\Delta$, and let $E \subseteq X$ be a prime divisor appearing in $f^*D$.
We have to show that the coefficient of $E$ in $f^*D$ is at least $m$.
Let $\eta_E$ be the generic point of $E$ and let $\xi := f(\eta_E)$ be its image in $Y$.
To verify the multiplicity condition for $E$, we may replace $Y$ by an open neighborhood of $\xi$ and thus, we may assume that $D$ is a principal divisor.
Let $y$ be a regular function on $Y$ whose vanishing locus is $D$.
Now consider the symmetric differential form $\frac{dy^{\otimes N}}{y^{N(1-1/m)}}$ on $Y$ (our assumption on $N$ ensures that $N(1-1/m)$ is an integer).
It defines a global section of $S^N \Omega^1_{(Y, \Delta)}$.
Thus, by assumption, the pullback $f^* \frac{dy^{\otimes N}}{y^{N(1-1/m)}}$ is regular on $X$.
Choose a uniformizer $t$ of the discrete valuation ring $\mathcal{O}_{X,\eta_E}$ and write $f(y) = ut^k$, with $u \in \mathcal{O}_{X,\eta_E}^{\times}$ and $f(y)$ denoting the image of $y$ under the ring morphism $\mathcal{O}_{Y, \xi} \to \mathcal{O}_{X,\eta_E}$ induced by $f$.
Now we calculate.
\[ f^* \frac{dy^{\otimes N}}{y^{N(1-1/m)}} = \frac{d(ut^k)^{\otimes N}}{(ut^k)^{N(1-1/m)}} = \frac{(ut^{k-1}dt+t^k du)^{\otimes N}}{u^{N(1-1/m)}~t^{kN(1-1/m)}} \]
The unique $dt^{\otimes N}$-term in the numerator has coefficient $u^Nt^{(k-1)N}$, whereas the denominator contains a $t$-power with exponent $kN(1-\frac{1}{m})$.
Thus, to make this expression be a regular form on $X$, we must have $(k-1)N \geq kN(1-\frac{1}{m})$, which is equivalent to $k \geq m$.
This means that the function $f^*y$ vanishes to order $\geq m$ at $\eta_E$, which is exactly what we want.
\end{proof}

In their extension of the theory of C-pairs to singular spaces, Kebekus--Rousseau essentially turn Lemma~\ref{sym_diffs_detect} into a definition; see \cite{KebekusRousseau1} for details.

On the other hand, for sufficiently well-behaved morphisms $f$, the symmetric differentials on the orbifold base $(Y, \Delta_f)$ pull back exactly as one would hope.
(Recall that a dominant morphism $X \to Y$ of smooth varieties is \emph{neat} if there is a proper birational morphism $X \to X'$ with $X'$ smooth such that every divisor contracted by $X \to Y$ is also contracted by $X \to X'$.)

\begin{lemma} \label{sym_diffs_pull_back}
Let $f \colon X \to Y$ be a dominant morphism of smooth varieties.
Assume that $Y$ is quasi-projective, $f$ is neat, and the orbifold base divisor $\Delta_f$ has snc support.
Then, for every integer $n \geq 1$, pullback of symmetric differentials induces a morphism $f^* S^n \Omega^1_{(Y, \Delta_f)} \to \Sym^n \Omega^1_X$.
\end{lemma}
\begin{proof}
By assumption, there exists a proper birational morphism $p \colon X \to X'$ with $X'$ smooth such that every prime divisor contracted by $f$ is also contracted by $p$.
Now let $U \subseteq Y$ be an open subset and let $\omega \in S^n \Omega^1_{(Y, \Delta_f)}(U)$ be a section.
We have to show that $f^*\omega$ is a regular symmetric $1$-form on $f^{-1}(U)$.
Since $X$ is smooth, the sheaf $\Sym^n \Omega^1_X$ is locally free, and thus it suffices to verify the regularity of $f^*\omega$ at all codimension one points of $f^{-1}(U)$ (Hartogs' lemma).

Consider first a codimension one point $\eta \in f^{-1}(U)$ such that $\xi = f(\eta)$ is of codimension $\leq 1$ in $Y$.
Then there is an open neighbourhood $V \subseteq f^{-1}(U)$ of $\eta$ such that the restricted morphism $V \to Y$ does not contract any divisors.
Consequently, it defines a C-pair morphism $V \to (Y, \Delta_f)$.
It hence follows from Lemma~\ref{sym_diffs_pull_back_cpair_morphism} that $f^*\omega$ is regular on $V$ and in particular does not have a pole at $\eta$.

Now consider a codimension one point $\eta \in f^{-1}(U)$ such that $\xi = f(\eta)$ is of codimension $\geq 2$ in $Y$.
Let $\phi$ be a rational function on $Y$ such that $\phi \omega$ is regular at every codimension one point belonging to $Y \setminus U$ and such that $\phi$ has neither a pole nor a zero in a neighborhood of $\xi \in Y$ (such a function exists since $Y$ is quasi-projective).
Since $\phi$ has neither a pole nor a zero in a neighborhood around $\xi \in Y$, checking that $f^*\omega$ does not have a pole at $\eta$ is equivalent to checking that $f^*(\phi \omega)$ does not have a pole at $\eta$.
As we assumed that $\eta$ is contracted by $f$, it is also contracted by $p$.
So $p(\eta)$ has codimension at least two in $X'$.
The two varieties $X$ and $X'$ being birational, we may view $f^*(\phi \omega)$ as a rational pluridifferential form on $X'$.
If it is regular at $p(\eta) \in X'$ (viewed as a form on $X'$), it is also regular at $\eta \in X$ (viewed as a form on $X$).
Since $X'$ is smooth, poles of pluridifferential forms occur purely in codimension one (again Hartogs' lemma), so that if $f^*(\phi \omega)$ is not regular at $\eta \in X$, there is a codimension one point $\tilde{\eta} \in X'$, specializing to $p(\eta) \in X'$, such that $f^*(\phi \omega)$ is not regular at $\tilde{\eta}$ either.
The point $\tilde{\eta} \in X'$ is a codimension one point, so that, as $p$ is an isomorphism outside a codimension two subset on $X'$, the point $\tilde{\eta}$ may be viewed as a codimension one point on $X$.
This point is visibly not contracted along the map $p$, so it is also not contracted along the map $f$.
Hence, by the second paragraph of this proof, $f^*(\phi \omega)$ having a pole at $\tilde{\eta}$ means that $\tilde{\eta} \notin f^{-1}(U)$.
However, $\phi \omega$ was assumed to not have any poles at any codimension one point belonging to $Y \setminus U$, so that it cannot have a pole at $f(\tilde{\eta})$.
Consequently, $f^*(\phi \omega)$ cannot have a pole at $\tilde{\eta}$; this implies that $\tilde{\eta}$ cannot exist, hence that $f^*(\phi \omega)$ must have been regular at $\eta \in X$, and hence that $f^*\omega$ itself is regular at $\eta \in X$.
This is what we wanted to show.
\end{proof}

The reliance of the proof on the auxiliary function $\phi$ is a bit strange, but it seems difficult to make the argument work without it.
If this reliance turns out to be unnecessary, the quasi-projectivity assumption on $Y$ can be dropped.
However, this assumption is harmless for our intended application.
Observe also that if one only wants to get a pullback morphism $\HH^0(Y, S^n \Omega^1_{(Y, \Delta_f)}) \to \HH^0(X, \Sym^n \Omega^1_X)$, the yoga involving $\phi$ is unneeded.
Thus, the morphism on $\HH^0$ exists without any quasi-projectivity assumptions.

Theorem~\ref{mainthmb} now follows immediately by combining Lemma~\ref{sym_diffs_detect} with Lemma~\ref{sym_diffs_pull_back}.

\section{Applications} \label{sect:applications}

In this section, we use Theorem~\ref{mainthmb} to prove that Campana's conjecture that every variety admitting a dense entire curve is Campana-special is implied by a C-pair generalization of the weak Green--Griffiths--Lang conjecture.
We also prove the analogous result for varieties admitting a potentially dense set of integral points and the weak Bombieri--Lang conjecture (below, we will make precise which conjectures we need).\smallbreak

Before addressing these conjectures however, we make some remarks on rational maps of general type and Campana-special varieties.
Let $X$ be a smooth variety.
Given a dominant rational map $f \colon X \ratmap Y$ to a smooth variety $Y$, we can define its orbifold base divisor $\Delta_f$ by precomposing with a proper birational map $p \colon X' \to X$ such that $f \circ p$ is a morphism and setting $\Delta_f := \Delta_{f \circ p}$.
An easy calculation shows that this is in fact independent of the choice of $p$ and thus well-defined \cite[Lemma~2.3]{SeveriPaper}.
We would then like to say that the \emph{Kodaira dimension} of $f$ is the Iitaka dimension of the divisor $K_Y + \Delta_f$; however, this turns out to be unstable under birational modifications in $Y$ (see e.g.\ \cite[Example~1.11]{CampanaFourier}).
So instead, we follow Campana and make the following definition.

\begin{definition} \label{def:kodaira_dim}
Let $f \colon X \ratmap Y$ be a dominant rational map of smooth varieties.
The \emph{Kodaira dimension} of $f$ is the minimum of the Iitaka dimensions of $K_{Y'} + \Delta_{f'}$, taken over all smooth projective varieties $Y'$ birational to $Y$.
(Here, we wrote $f'$ for the dominant rational map $X \ratmap Y'$ which generically agrees with $f$.)
We say that $f$ is of \emph{general type} if its Kodaira dimension is $\dim Y$.
\end{definition}

The minimum in Definition~\ref{def:kodaira_dim} is achieved if $f$ is neat and the orbifold base divisor $\Delta_f$ has snc support \cite[Remark~2.11]{SeveriPaper}; if $X$ is proper, the condition on $\supp \Delta_f$ can be dropped.
The Kodaira dimension of a map $f$ has the interpretation of measuring the number of top degree pluridifferential forms on $Y$, perhaps with poles, whose pullback to $X$ is regular with at most logarithmic poles at infinity; see \cite[Proposition~2.9]{SeveriPaper} for a precise statement.
We can now make the following definition.

\begin{definition} \label{def:special}
A smooth variety $X$ is \emph{Campana-special} if it does not admit a rational map of general type whose image is positive-dimensional.
\end{definition}

Definition~\ref{def:special} coincides with the definition given in the introduction since every morphism can be birationally modified to a neat morphism; see the proof of Lemma~\ref{lemma:folklore} below.\smallbreak

Coming back to C-pairs, we say that a smooth C-pair $(X, \Delta)$ with $X$ proper is of \emph{general type} if the $\mathbb{Q}$-divisor $K_X + \Delta$ is big.
Thus, the definition of a Campana-special variety strongly suggests that a non-Campana-special variety admits, perhaps up to birational modification, a dominant morphism to a smooth C-pair of general type.
This is indeed true, but it is not a formal consequence of the definitions, as the examples in Section~\ref{sect:examples} show.
Nonetheless, this has been claimed many times in the literature, to the point that one can almost consider it to be ``folklore knowledge''.
The source of this claim is perhaps simply an overly optimistic reading of the results in \cite{CampanaFourier} and \cite{CampanaJussieu}, particularly \cite[Proposition~4.10]{CampanaJussieu}.
Hence we explicitly stress that the result which we now state is not a special case of \cite[Proposition~4.10]{CampanaJussieu}.

\begin{lemma} \label{lemma:folklore}
Let $X$ be a smooth variety which is not Campana-special.
Then there are a proper birational morphism $\widetilde{X} \to X$ with $\widetilde{X}$ smooth and a smooth C-pair $(Y, \Delta)$ of general type with $Y$ projective such that there is a dominant C-pair morphism $\widetilde{X} \to (Y, \Delta)$.
\end{lemma}
\begin{proof}
Since $X$ is not Campana-special, it admits a fibration of general type $X \ratmap Y$ with $Y$ a smooth proper variety (e.g.\ the core of $X$).
Choose a proper birational morphism $X' \to X$ with $X'$ smooth such that $f \colon X' \to Y$ is a morphism.
By Raynaud--Gruson flattening \cite[Theorem~5.2.2]{RaynaudGruson}, there is a proper birational morphism $Y' \to Y$ such that if $X''$ denotes the main component of $X' \times_Y Y'$, the induced morphism $X'' \to Y'$ is flat.
Blowing up $Y'$ further, we may additionally assume that $Y'$ is smooth projective and that the total transform of $\supp \Delta_f$ is an snc divisor on $Y'$.
Let $\widetilde{X}$ be a resolution of singularities of $X''$.
We claim that $g \colon \widetilde{X} \to Y'$ is neat.
Indeed, every divisor contracted by $g$ must already be contracted by $\widetilde{X} \to X''$ since $X'' \to Y'$ is flat.
In particular, every divisor contracted by $g$ is also contracted by $\widetilde{X} \to X'$, proving the neatness of $g$.
Moreover, as $g \colon \widetilde{X} \to Y'$ is a birational model of the morphism $f \colon X' \to Y$, it is easy to see that the support of $\Delta_g$ is contained in the total transform of the support of $\Delta_f$.
In particular, $\Delta_g$ is supported on an snc divisor.
Thus, by Theorem~\ref{mainthmb}, the morphism $g$ is a C-pair morphism $\widetilde{X} \to (Y', \Delta_g)$.
As $X \ratmap Y$ is a fibration of general type, the pair $(Y', \Delta_g)$ is of general type, so we are done. 
\end{proof}

Using Lemma~\ref{lemma:folklore}, the proof of Theorems~\ref{mainthmc} and \ref{mainthmd} is now standard.
But let us be precise about which conjectures we exactly need.

\subsection{Dense entire curves}

We start with the proof of Theorem~\ref{mainthmc}, concerning dense entire curves.
First observe that, over the complex numbers, the definition of a C-pair morphism given in the introduction adapts verbatim to give a notion of morphisms from a complex manifold into a C-pair.
Thus, we can make the following definition.

\begin{definition}
Let $(X, \Delta)$ be a C-pair over $\mathbb{C}$.
An \emph{entire curve} is a holomorphic C-pair map $\mathbb{C} \to (X, \Delta)^{\an}$.
A \emph{dense entire curve} is an entire curve whose image is Zariski-dense in $X$.
\end{definition}

The Green--Griffiths--Lang conjecture asserts that given a smooth projective variety $X$ of general type over the complex numbers, there is a proper Zariski-closed subset $Z \subsetneq X$ such that every nonconstant entire curve factors over $Z$.
As a consequence, it predicts that a positive-dimensional such $X$ does not admit a dense entire curve; this specific consequence is however weaker than the ``full'' Green--Griffiths--Lang conjecture, so that we call it the \emph{weak Green--Griffiths--Lang conjecture}. 
The weak Green--Griffiths--Lang conjecture for C-pairs now reads as follows; it is a special case of \cite[Conjecture~13.17]{CampanaJussieu}. 

\begin{conjecture}[Weak Green--Griffiths--Lang conjecture for C-pairs]
Let $(X, \Delta)$ be a smooth C-pair of general type over the complex numbers with $X$ projective and positive-dimensional. 
Then $(X, \Delta)$ does not admit a dense entire curve.
\end{conjecture}

A smooth variety $X$ is said to be of \emph{log-general type} if for one (then every) compactification $X \subseteq \overline{X}$ whose boundary $D := \overline{X} \setminus X$ is an snc divisor, the divisor $K_{\overline{X}} + D$ is big.
In this case, the pair $(\overline{X}, D)$ is a C-pair of general type, so that the weak Green--Griffiths--Lang conjecture predicts that $(\overline{X}, D)$ does not admit a dense entire curve.
Equivalently, $X$ does not admit a dense entire curve.

\begin{proof}[Proof of Theorem~\ref{mainthmc}]
A dense entire curve on $X$ lifts to a dense entire curve on any resolution of $X$.
Hence we may assume that $X$ is smooth.
We now proceed to prove the result by contradiction, so assume that $X$ is not Campana-special.
By assumption, there is a holomorphic map $\mathbb{C} \to X^{\an}$ with Zariski-dense image.
If $X$ is of general type, this contradicts the weak Green--Griffiths--Lang conjecture for the variety $X$ and we are done.
So we may assume that $X$ is not of general type.
By Lemma~\ref{lemma:folklore}, there is a proper birational morphism $\widetilde{X} \to X$ with $\widetilde{X}$ smooth and a smooth C-pair $(Y, \Delta)$ of general type with $Y$ projective such that there is a dominant C-pair morphism $f \colon \widetilde{X} \to (Y, \Delta)$.
If $\dim Y = \dim \widetilde{X}$, the Riemann--Hurwitz formula implies that $f^*(K_Y + \Delta) \leq K_{\widetilde{X}}$, so that the Iitaka dimension of $K_Y + \Delta$ is at most that of $K_{\widetilde{X}}$, which is the Kodaira dimension of $\widetilde{X}$.
Thus, since $\widetilde{X}$ is not of general type, we must have $\dim Y < \dim \widetilde{X}$.
The entire curve $\mathbb{C} \to X^{\an}$ lifts along the proper birational morphism $\widetilde{X} \to X$ to an entire curve $\mathbb{C} \to \widetilde{X}^{\an}$, whose image is still Zariski-dense.
Composing, we obtain a dense entire curve in $(Y, \Delta)^{\an}$, contradicting the weak Green--Griffiths--Lang conjecture for the C-pair $(Y, \Delta)$.
\end{proof}

We note that if one restricts to proper $X$ of dimension $\leq n$ in the statement of Theorem~\ref{mainthmc}, one gets away with only assuming the weak Green--Griffiths--Lang conjecture for smooth projective varieties $X$ of dimension $\leq n$.
However, one still needs to assume weak Green--Griffiths--Lang for all C-pairs of dimension $\leq n-1$, which then also includes assuming it for all quasi-projective varieties of dimension $\leq n-1$.
(It might seem that one can get away with only assuming it for C-pairs with \emph{finite} multiplicities, but this case is easily seen to imply the conjecture for all C-pairs; see \cite[Section~6.1]{LafonThreefoldPaper} for details.)

\subsection{Density of integral points}

Our strategy for proving the number-theoretic Theorem~\ref{mainthmd} is similar to that for proving the complex-analytic Theorem~\ref{mainthmc}.
However, to state the conjectures precisely, we need to discuss integral points on varieties and C-pairs.
If $X$ is a variety over $\overline{\mathbb{Q}}$ and $R \subseteq \overline{\mathbb{Q}}$ is a subring, a \emph{model} for $X$ over $R$ is an $R$-scheme $\mathcal{X}$ equipped with an isomorphism $\mathcal{X} \otimes_R \overline{\mathbb{Q}} \cong X$.

\begin{definition}
Let $X$ be a variety over a subfield of $\overline{\mathbb{Q}}$.
We say that $X$ \emph{satisfies potential density of integral points} if there exist a number field $K$, a finite set $S$ of finite places of $K$, and a finite type separated model $\mathcal{X}$ for $X$ over $\mathcal{O}_{K,S}$ such that $\mathcal{X}(\mathcal{O}_{K,S}) \subseteq X(\overline{\mathbb{Q}})$ is Zariski-dense in $X$.
\end{definition}

If $X$ is proper, the model $\mathcal{X}$ can be chosen to be proper.
In that case, we have $\mathcal{X}(\mathcal{O}_{K,S}) = \mathcal{X}(K)$, so that the potential density of integral points is equivalent to that of the rational points on $X$.
The definition of potential density for C-pairs is slightly more involved.

\begin{definition}
Let $(X, \Delta)$ be a C-pair over a subfield of $\overline{\mathbb{Q}}$, let $K$ be a number field, and let $S$ be a finite set of finite places of $K$.
Let $\mathcal{X}$ be a model for $X$ over $\mathcal{O}_{K,S}$ and let $\Delta_{\mathcal{X}}$ be the closure of $\Delta$ in $\mathcal{X}$.
Write $\Delta_{\mathcal{X}} = \sum_i (1-\frac{1}{m_i}) D_i$.
Then a morphism $f \colon \Spec \mathcal{O}_{K,S} \to \mathcal{X}$ is an \emph{integral point} of $(\mathcal{X}, \Delta_{\mathcal{X}})$ if it is disjoint from $\floor{\Delta_{\mathcal{X}}}$ and for every $i$, either $f$ factors over $D_i$ or every coefficient of $f^*D_i$ is at least $m_i$.
We say that $(X, \Delta)$ satisfies \emph{potential density of integral points} if for some choice of $K$, $S$ and $\mathcal{X}$ as above, the integral points of $(\mathcal{X}, \Delta_{\mathcal{X}})$ are Zariski-dense in $X$.
\end{definition}

The reason for this definition is that it is precisely what a ``C-pair morphism'' $\Spec \mathcal{O}_{K,S} \to (\mathcal{X}, \Delta_{\mathcal{X}})$ should be (except that we have not defined C-pair morphisms in this generality).
The same calculation showing that C-pair morphisms can be composed then easily yields the following lemma.

\begin{lemma} \label{lemma:descend_pot_density}
Let $X$ be a smooth variety, let $(Y, \Delta)$ be a C-pair, and let $f \colon X \to (Y, \Delta)$ be a dominant C-pair morphism.
Suppose that $X$ satisfies potential density of integral points.
Then $(Y, \Delta)$ satisfies potential density of integral points.
\end{lemma}

The Bombieri--Lang conjecture asserts that given a smooth projective variety $X$ of general type over a number field $K$, there is a proper closed subset $Z \subsetneq X$ such that for all number fields $L$ extending $K$, all but finitely many $L$-points of $X$ lie in $Z$.
As a consequence, it predicts that a positive-dimensional such $X$ does not have a potentially dense set of rational points; this specific consequence is however weaker than the ``full'' Bombieri--Lang conjecture, so that we call it the \emph{weak Bombieri--Lang conjecture}. 
The weak Bombieri--Lang conjecture for C-pairs now reads as follows; it is a special case of \cite[Conjecture~13.23]{CampanaJussieu}.

\begin{conjecture}[Weak Bombieri--Lang conjecture for C-pairs]
Let $(X, \Delta)$ be a smooth C-pair of general type over $\overline{\mathbb{Q}}$ with $X$ projective and positive-dimensional. 
Then $(X, \Delta)$ does not satisfy potential density of integral points.
\end{conjecture}

Similarly to the case of the weak Green--Griffiths--Lang conjecture, the Bombieri--Lang conjecture for C-pairs implies in particular that given a smooth variety $X$ of log-general type, it does not satisfy potential density of integral points.

\begin{proof}[Proof of Theorem~\ref{mainthmd}]
Potential density of integral points is invariant under proper birational morphisms, so we may assume that $X$ is smooth.
We now proceed to prove the result by contradiction, so assume that $X$ is not Campana-special.
By assumption, $X$ has a potentially dense set of integral points.
If $X$ is of general type, this directly contradicts the weak Bombieri--Lang conjecture for the variety $X$ and we are done.
So we may assume that $X$ is not of general type.
By Lemma~\ref{lemma:folklore}, there is a proper birational morphism $\widetilde{X} \to X$ with $\widetilde{X}$ smooth and a smooth C-pair $(Y, \Delta)$ of general type with $Y$ projective such that there is a dominant C-pair morphism $f \colon \widetilde{X} \to (Y, \Delta)$.
If $\dim Y = \dim \widetilde{X}$, the Riemann--Hurwitz formula implies that $f^*(K_Y + \Delta) \leq K_{\widetilde{X}}$, so that the Iitaka dimension of $K_Y + \Delta$ is at most that of $K_{\widetilde{X}}$, which is the Kodaira dimension of $\widetilde{X}$.
Thus, since $\widetilde{X}$ is not of general type, we must have $\dim Y < \dim \widetilde{X}$.
Again using the proper-birational invariance of potential density of integral points, we see that $\widetilde{X}$ has a potentially dense set of integral points.
Using Lemma~\ref{lemma:descend_pot_density}, we see that $(Y, \Delta)$ satisfies potential density of integral points, contradicting the weak Bombieri--Lang conjecture for the C-pair $(Y, \Delta)$.
\end{proof}

As with Theorem~\ref{mainthmc}, if one restricts to proper $X$ of dimension $\leq n$ in the statement of Theorem~\ref{mainthmd}, one gets away with only assuming the weak Bombieri--Lang conjecture for smooth projective varieties $X$ of dimension $\leq n$.
However, one still needs to assume weak Bombieri--Lang for all C-pairs of dimension $\leq n-1$, which then also includes assuming it for all quasi-projective varieties of dimension $\leq n-1$.

\bibliographystyle{alpha}
\bibliography{orbifoldbase}

\end{document}